\documentclass{article}

\usepackage{setspace}
\usepackage{amssymb,amsmath,amsthm,tabu,longtable}
\usepackage{subfigure}
\usepackage{calc}
\usepackage{verbatim}
\usepackage{enumerate}
\usepackage{epsfig}
\usepackage{graphicx}
\usepackage{graphics}
\usepackage {tikz}
\usetikzlibrary{automata,arrows,calc,positioning}
\usetikzlibrary{decorations.pathreplacing}
\usepackage{xcolor}\usepackage[hmargin=3cm,vmargin=3.5cm]{geometry}
\usepackage{chngcntr}
\counterwithin{figure}{section}

\newtheorem{definition}{Definition}[section]
\newtheorem{lemma}[definition]{Lemma}
\newtheorem{proposition}[definition]{Proposition}
\newtheorem{theorem}[definition]{Theorem}
\newtheorem{corollary}[definition]{Corollary}

\newtheorem{conjecture}[definition]{Conjecture}

                                {\null\hfill$\Box$\par\medskip\medskip\medskip\medskip}

\newcommand*{\QED}{\null\hfill$\Box$\par\medskip}%

\makeatletter
\renewcommand{\pod}[1]{\mathchoice
  {\allowbreak \if@display \mkern 18mu\else \mkern 8mu\fi (#1)}
  {\allowbreak \if@display \mkern 18mu\else \mkern 8mu\fi (#1)}
  {\mkern4mu(#1)}
  {\mkern4mu(#1)}
}

\newcommand{\avd}{\text{\rm{avd}}}

\newcounter{eqcount}

\title{On the Unimodality of Domination Polynomials}

\author{Iain Beaton\thanks{Corresponding author}\\
\small Department of Mathematics \& Statistics\\[-0.8ex]
\small Dalhousie University\\[-0.8ex] 
\small Halifax, CA\\
\small\tt ibeaton@dal.ca\\
\and
Jason I. Brown\thanks{Supported by NSERC grant RGPIN-2018-05227}\\
\small Department of Mathematics \& Statistics\\[-0.8ex]
\small Dalhousie University\\[-0.8ex] 
\small Halifax, CA\\
\small\tt Jason.Brown@dal.ca\\
}

\begin{document}

\tikzset{bignode/.style={minimum size=3em,}}
\maketitle

\begin{abstract}
A polynomial is said to be unimodal if its coefficients are non-decreasing and then non-increasing. The domination polynomial of a graph $G$ is the generating function of the number of domination sets of each cardinality in $G$, and its coefficients have been conjectured to be unimodal. In this paper we will show the domination polynomial of paths, cycles and complete multipartite graphs are unimodal, and that the  domination polynomial of almost every graph is unimodal with mode $ \lceil \frac{n}{2}\rceil $.
\end{abstract}

\setstretch{1.4}

\section{Introduction}\label{sec:intro}

Domination in graphs has been investigated both for applied and theoretical reasons. A subset of vertices $S$ of a (finite, undirected) graph $G=(V,E)$ is a {\em dominating set} iff every vertex of $G$ is either in $S$ or adjacent to a vertex of $S$ (equivalently, for any vertex $v$ of $G$, the {\em closed neighbourhood} $N[v]$ of $v$ has nonempty intersection with $S$).  Much of the attention has been directed at the \emph{domination number} of $G$, $\gamma(G)$, the minimum cardinality of a dominating set of $G$, but overall, the study of dominating sets in graphs is quite extensive (see, for example, \cite{hedetniemi}).  

As for many graph properties, one can more thoroughly examine domination via generating functions. Let $d_i(G)$ denote be the number of dominating sets of a graph $G$ of cardinality $i$. The \emph{domination polynomial} $D(G,x)$ of $G$ is defined as

$$D(G,x) = \sum_{i=\gamma (G)}^{|V(G)|} d_i(G)x^i.$$

\noindent (See \cite{2012AlikhaniPHD}, for example, for a thorough discussion of domination polynomials.) 

A natural question for any graphs polynomial is whether or not the sequence of coefficients is unimodal: a polynomial with real coefficients $a_0 + a_1x + \cdots + a_nx^n$ is said to be \emph{unimodal} if there exists $0 \leq k \leq n$, such that

$$a_0 \leq \cdots \leq a_{k-1} \leq a_k \geq a_{k-1} \geq \cdots \geq a_{n}$$

\noindent (in such a case, we call the location(s) of the largest coefficient the {\em mode}).
To show a polynomial is unimodal, it has often been helpful (and easier) to show a stronger condition, called log-concavity, holds, as the latter does not require knowing where the peak might be located. A polynomial is \emph{log-concave} if for every $1 \leq i \leq n-1$, $a_i^2 \geq a_{i-1}a_{i+1}$. It is not hard to see that a polynomial with positive coefficients that is log-concave is also unimodal. 
%
%

A variety of techniques have been used to show many graph polynomials are log-concave, and hence unimodal, including:
\begin{itemize}
\item real analysis  (log-concavity of the matching polynomial \cite{1972Heilmann} and the independence polynomial of claw-free graphs \cite{2007Chudnovsky}),
\item homological algebra (June Huh's proof of the log concavity of chromatic polynomials), and
\item combinatorial arguments (the arguments of Krattenthaler \cite{1996Krattenthaler} and Hamidoune \cite{1990Hamidoune} that reproved the log concavity of matching polynomials and independence polynomial of claw-free graphs, respectively, as well as Horrocks' \cite{2002Horrocks} result that the dependent k-set polynomial is log-concave (a subset of vertices is {\em dependent} iff it contains an edge of the graph). 
\end{itemize}

The domination polynomial of every graph of order at most 8 is log-concave. However the domination polynomial of the graph on 9 vertices in Figure \ref{fig:DomPolynLC} is

\vspace{-2mm}
$$D(G,x) = x^9+9x^8+35x^7+75x^6+89x^5+50x^4+7x^3+x^2$$

\noindent which is not log-concave as $d_3(G)^2 = 49$ but $d_4(G)d_2(G) = 50$. 
Although not all domination polynomials are log-concave they are conjectured to be unimodal \cite{IntroDomPoly2014}.

\begin{center}
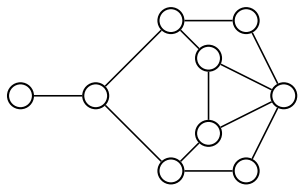
\begin{figure}[!h]
\begin{center}
\begin{tikzpicture}[> = stealth,semithick]
\begin{scope}[every node/.style={circle,thick,draw}]
    \node (1) at (0,0) {};
    \node (2) at (1,0) {};   
    \node (3) at (2,1) {};   
    \node (4) at (2,-1) {};   
    \node (5) at (2.5,0.5) {};
    \node (6) at (2.5,-0.5) {};
    \node (7) at (3,1) {};   
    \node (8) at (3,-1) {}; 
    \node (9) at (3.5,0) {}; 
    
\end{scope}

\begin{scope}
    \path [-] (1) edge node {} (2);
    \path [-] (2) edge node {} (3);
    \path [-] (2) edge node {} (4);
    \path [-] (3) edge node {} (5); 
    \path [-] (3) edge node {} (7);
    \path [-] (4) edge node {} (6);
    \path [-] (4) edge node {} (8);
    \path [-] (5) edge node {} (6);
    \path [-] (5) edge node {} (9);
    \path [-] (6) edge node {} (9);
    \path [-] (7) edge node {} (9);
    \path [-] (8) edge node {} (9);
\end{scope}
\end{tikzpicture}
\end{center}
\caption{The only graph of order 9 which is not log-concave}%
\label{fig:DomPolynLC}%
\end{figure}
\end{center}

\begin{conjecture}
\label{conj:unimodal}
\textnormal{\cite{IntroDomPoly2014}} The domination polynomial of any graph is unimodal.
\end{conjecture}

\noindent To date, only a little progress has been made on Conjecture \ref{conj:unimodal}. 
In the following theorem, $kG$ denotes the disjoint union of $k$ copies of $G$, and $G \circ H$ denotes the \emph{corona} \cite{1970Corona} of two disjoint graphs $G$ and $H$ is formed from $G$ and $|V(G)|$ copies of $H$, one for each vertex of $G$, by joining $v \in V(G)$ to every vertex in the corresponding copy of $H$.

\begin{theorem}
\label{thm:DomPolyFamUnimodal}
\textnormal{\cite{2014DomFamUnimodal}} For $n \geq 1$ and any graph $G$:

 \renewcommand{\labelenumi}{(\roman{enumi})}
 \begin{enumerate}
   \item The friendship graph $F_n \cong K_1 \vee nK_2$ is unimodal.
   \item The graph formed by adding a universal vertex to $nK_2 \cup K_1$ is unimodal.
   \item $G \circ K_n$ is log-concave and hence unimodal.
   \item $G \circ P_3$ is log-concave and hence unimodal.

 \end{enumerate}
\end{theorem}

\noindent In this paper we extend the families for which unimodality of the domination polynomial is known to paths, cycles and complete multipartite graphs. More significantly, we will also show almost all domination polynomials are unimodal with mode $ \lceil \frac{n}{2}\rceil $.

\section{Paths, Cycles and Complete Multipartite Graphs}

We say a graph contains a \emph{simple $k$-path} if there exists $k$ vertices of degree two which induce a path in $G$. Two families of graphs which contains simple $k$-paths are paths $P_n$ and cycles $C_n$ (where $k = n-2$ and $n-1$, respectively). 

\begin{theorem}
\label{thm:DomPolyIP}
\textnormal{\cite{2012Recurr}} Suppose $G$ is a graph with vertices $u, v, w$ that form a simple 3-path. Then 
$$D(G,x)=x(D(G/u,x) + D(G/u/v,x) + D((G/u/v/w,x))$$
\noindent where $G/u$ is the graph formed by joining every pair of neighbours of $u$ and then deleting $u$ and $G/u/v = (G/u)/v$.
\QED
\end{theorem}

There is no useful closed formula for the coefficients of $D(P_n,x)$ and $D(C_n,x)$. However consider Table~\ref{tab:PnCn}, which displays $D(P_n,x)$, $D(C_n,x)$, and their respective modes.
Note that for both paths and cycles, consecutive modes differ by at most one in these small cases. We will now show that these observations for small $n$ are sufficient to prove that the domination polynomials of all paths and cycles are unimodal.

\begin{theorem}
\label{thm:PCunimodal}
Suppose we have a sequence of polynomials $(f_n)_{n \geq 1}$ with non-negative coefficients which satisfy

\[f_n=x(f_{n-1} + f_{n-2} + f_{n-3}) \refstepcounter{eqcount} \label{eqn:recurr} \tag{\theeqcount}\]

\noindent for $n \geq 4$. Let $\mathcal{P}_n$ denote the property that for all $i \in \{1,2,\ldots,n\}$, $f_i$ is unimodal with mode $m_i$ and if $i \geq 2$, $0 \leq m_i-m_{i-1} \leq 1$. 
Assume $\mathcal{P}_4$ holds. Then $\mathcal{P}_n$ holds for all $n \geq 1$ (and so each $f_n$ is unimodal).
\end{theorem}

\begin{table}
\begin{center}
\begin{tabular}{c|c|c}
$n$ & $D(P_n,x)$ & $m_n$ \\ \hline
$1$ & $x$ & $1$ \\ \hline
$2$ & $x^2+2x$ & $1$ \\ \hline
$3$ & $x^3+3x^2+x$ & $2$ \\ \hline
$4$ & $x^4+4x^3+4x^2$ & $3$ 
\end{tabular}
\qquad
\begin{tabular}{c|c|c}
$n$ & $D(C_n,x)$ & $m_n$ \\ \hline
$3$ & $x^3+3x^2+3x$ & $2$ \\ \hline
$4$ & $x^4+4x^3+6x^2$ & $2$ \\ \hline
$5$ & $x^5+5x^4+10x^3+5x^2$ & $3$ \\ \hline
$6$ & $x^6+6x^5+15x^4+14x^3+3x^2$ & $4$ 
\end{tabular}
\caption{Domination polynomials for paths and cycle of small order}%
\label{tab:PnCn}
\end{center}
\end{table}

\begin{proof}
We will prove our assertion via induction on $n \geq 4$. Our base case is satisfied by the assumption that $\mathcal{P}_4$ holds. For some $k \geq 4$, suppose $\mathcal{P}_k$ holds, and so $\mathcal{P}_j$ holds for all $1 \leq j \leq k$. To show $\mathcal{P}_{k+1}$ holds it suffices to show $f_{k+1}$ is unimodal with mode $m_{k+1}=m_k \text{ or } m_k+1$. By our inductive hypothesis, $f_{k}$, $f_{k-1}$, and $f_{k-2}$ are all unimodal with modes $m_k$, $m_{k-1}$, and $m_{k-2}$ respectively. Additionally, $m_{k-1} \leq m_{k} \leq m_{k-1}+1$ and $m_{k-2} \leq m_{k-1} \leq m_{k-2}+1$. For simplicity let $m_k=m$. Note that $m -2 \leq m_{k-2} \leq m_{k-1} \leq m_k = m$. Furthermore for each $n \geq 1$ let

$$f_{n}=\sum_{j=0}^{\infty} a_{n,j}x^j.$$

\noindent Therefore for $n=k, k-1, k-2$ we have

$$a_{n,0} \leq a_{n,1} \leq \cdots \leq a_{n,m-2} \text{ and } a_{n,m} \geq a_{n,m+1} \geq \cdots.$$

\noindent By the recursive relation $(\ref{eqn:recurr})$ we see that $a_{k+1,0} =0$ and for each $j \geq 1$

$$a_{k+1,j} = a_{k,j-1}+a_{k-1,j-1}+a_{k-2,j-1}.$$

\noindent Therefore

$$0=a_{k+1,0} \leq a_{k+1,1} \leq \cdots \leq a_{k+1,m-1} \text{ and } a_{k+1,m+1} \geq a_{k+1,m+2} \geq \cdots.$$ 

\noindent We will now show $a_{k+1,m-1} \leq a_{k+1,m}$. Consider the following two cases:

\vspace{2mm}
\noindent \textbf{Case 1: $m -1 \leq m_{k-2}\leq m$}
\vspace{1mm}

As $m -1 \leq m_{k-2}$ then the modes of $f_{k}$, $f_{k-1}$, and $f_{k-2}$ are each at least $m-1$. Thus $a_{k,m-2} \leq a_{k,m-1}$, $a_{k-1,m-2} \leq a_{k-1,m-1}$, and $a_{k-2,m-2} \leq a_{k-2,m-1}$. Therefore 

\begin{align*}
a_{k+1,m-1} &= a_{k,m-2}+a_{k-1,m-2}+a_{k-2,m-2} \\
           & \leq a_{k,m-1}+a_{k-1,m-1}+a_{k-2,m-1}  \\
           & = a_{k+1,m}.
\end{align*}


\vspace{2mm}
\noindent \textbf{Case 2: $m_{k-2} = m-2$}
\vspace{1mm}

By the recursive relation the polynomials follow we obtain $a_{k,0}=0$ and $a_{k,j} = a_{k-1,j-1}+a_{k-2,j-1}+a_{k-3,j-1}$ for each $j \geq 1$. Note $a_{k,m} \geq a_{k,m-1}$ because the mode of $f_k$ is $m$. Therefore 

$$a_{k-1,m-1}+a_{k-2,m-1}+a_{k-3,m-1} \geq a_{k-1,m-2}+a_{k-2,m-2}+a_{k-3,m-2}.$$

\noindent Let the mode of $f_{k-3}$ be $m_{k-3}$. By our inductive hypothesis $m_{k-3} \leq m_{k-2}=m-2$, and therefore $a_{k-3,m-1} \leq a_{k-3,m-2}$. Furthermore 

$$a_{k-1,m-1}+a_{k-2,m-1}\geq a_{k-1,m-2}+a_{k-2,m-2}.$$

\noindent Again the mode of $f_k$ is $m$ so $a_{k,m-1} \geq a_{k,m-2}$. Hence

\begin{align*}
a_{k+1,m-1} &= a_{k,m-2}+a_{k-1,m-2}+a_{k-2,m-2} \\
           &  \leq a_{k,m-1}+a_{k-1,m-1}+a_{k-2,m-1} \\
           &  = a_{k+1,m}.
\end{align*}


\vspace{2mm}

As $a_{k+1,m-1}\leq a_{k+1,m}$ then $f_{k+1}$ is unimodal with mode at either $m$ or $m+1$. Therefore $\mathcal{P}_{k+1}$ holds and by induction $\mathcal{P}_{n}$ holds for all $n \geq 1$.
\end{proof}

Note that for a vertex $u$ in either $P_n$ or $C_n$, $P_{n}/u \cong P_{n-1}$ and $C_{n}/u \cong C_{n-1}$.  Thus by Theorem \ref{thm:DomPolyIP} paths and cycles follow the recursion relation $(\ref{eqn:recurr})$. It follows from Theorem \ref{thm:PCunimodal} and Table \ref{tab:PnCn} that the following corollary holds.

\begin{corollary}
For $n \in \mathbb{N}$ and $n \geq 3$, $P_n$ and $C_n$ are unimodal. \QED
\end{corollary}

We remark that Theorem \ref{thm:PCunimodal} can be leveraged to show many other families of graphs which contain simple $k$-paths are unimodal. For example, let $L_n$ denote a path on $n-2$ vertices with a $K_2$ joined to one of the leaves (See Figure \ref{fig:Ln}).

\begin{center}
\begin{figure}[!h]
\begin{center}
\begin{tikzpicture}[> = stealth,semithick]
\begin{scope}[every node/.style={circle,thick,draw}]
    \node (1) at (0,-0.5) {};
    \node (2) at (0,0.5) {};   
    \node (3) at (1,0) {};   
    \node (4) at (2,0) {}; 
      
    \node (5) at (3.5,0) {};
    \node (6) at (4.5,0) {};

\end{scope}

\begin{scope}
    \path [-] (1) edge node {} (2);
    \path [-] (2) edge node {} (3);
    \path [-] (1) edge node {} (3);
    \path [-] (3) edge node {} (4);
    \path [-] (4) edge node {} (2.5,0);
    \path [-] (3,0) edge node {} (5); 
    \path [-] (5) edge node {} (6);
\end{scope}

\path (3,0) -- node[auto=false]{\ldots} (2.5,0);
\end{tikzpicture}
\end{center}
\caption{The graph $L_n$}%
\label{fig:Ln}%
\end{figure}
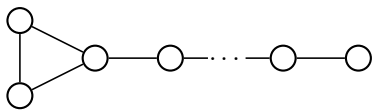
\end{center}

\noindent For $n \geq 5$, $L_n$ contains a simple $n-4$-path and therefore by Theorem \ref{thm:DomPolyIP} follows the recurrence relation $(\ref{eqn:recurr})$. Furthermore, Table~\ref{tab:Ln} shows that the base condition in Theorem~\ref{thm:PCunimodal} holds for four consecutive values of $n$ -- 4,5,6 and 7. It follows that $L_n$ is unimodal for $n\geq 4$. 

\begin{table}[!h]
\begin{center}
\begin{tabular}{c|c|c}
$n$ & $D(L_n,x)$ & $m_n$ \\ \hline
$4$ & $x^4+4x^3+5x^2+x$ & $2$ \\ \hline
$5$ & $x^5+5x^4+9x^3+6x^2$ & $3$ \\ \hline
$6$ & $x^6+6x^5+14x^4+14x^3+4x^2$ & $4$ \\ \hline
$7$ & $x^7+7x^6+20x^5+27x^4+15x^3+x^2$ & $4$  
\end{tabular}
\caption{Domination polynomials for graphs $L_n$.}%
\label{tab:Ln}
\end{center}
\end{table}

\vspace{0.2in}

We shall now show complete multipartite graphs are unimodal. 
We shall rely on an important result of  \noindent Alikhani et al. that shows that the coefficients of the domination polynomial are non-decreasing up to $\frac{n}{2}$.

\begin{proposition}
\textnormal{\cite{IntroDomPoly2014}} Let $G$ be a graph of order $n$. Then for every $0 \leq i < \frac{n}{2}$, we have $d_i(G) \leq d_{i+1}(G)$.
\label{thm:domincrease}
\end{proposition} 

We are now ready to proceed.

\begin{theorem}
\label{thm:multipartite}
For $n_1, \ldots, n_k \in \mathbb{N}$, the complete multipartite graph $K_{n_1, \ldots, n_k }$ is unimodal.
\end{theorem}

\begin{proof}
Set $G = K_{n_1, \ldots, n_k }$. Consider any subset of vertices $S \subseteq V(G)$ which is dependent. Therefore $S$ contains two adjacent vertices $u$ and $v$. Note that as $G$ is complete multipartite, each of $u$ and $v$ are adjacent to every vertex in $G$ except the other vertices in their respective parts. As $u$ and $v$ are adjacent, they are not in the same part of $G$ and hence $S$ dominate $G$. Let $f(x) = f_{G}(x)$ denote the dependent polynomial of $G$ (the generating function of  the number of dependent sets of cardinality $k$ in $G$). As mentioned earlier, $f(x)$ is log-concave \cite{2002Horrocks}. Furthermore 

$$D(G,x) = f(x)+\sum_{i=1}^k x^{n_i},$$

\noindent as the only dominating sets which are not dependent sets is all the vertices of a part of $G$. Let $G$ have $n$ vertices. By Proposition~\ref{thm:domincrease} $d_i(G) \leq d_{i+1}(G)$ for every $0 \leq i < \frac{n}{2}$. If all $n_j < \frac{n}{2}$, then $d_i(G)=f_i$ for all $i \geq \frac{n}{2}$ where $f_i$ is the coefficient of $x^i$ in $f(x)$. Furthermore as $f(x)$ is log-concave then $f(x)$ as unimodal and hence $D(G,x)$ is unimodal. So suppose there exists some $n_j \geq \frac{n}{2}$. Note that there is either exactly one $n_j \geq \frac{n}{2}$ or $G \cong K_{\frac{n}{2},\frac{n}{2}}$.

First suppose there is exactly one $n_j \geq \frac{n}{2}$. Then $d_i(G)=f_i$ for all $i \geq \frac{n}{2}$ except for $d_j(G)=f_j+1$. As the sequence $f(x)$ is log-concave and hence unimodal then the only way for the sequence to not be unimodal is for $f_j=f_{j+1}<f_{j+2}$ or $f_{j-2}>f_{j-1}=f_{j}$. However each case would contradict $f(x)$ being log-concave.

Now suppose $G \cong K_{\frac{n}{2},\frac{n}{2}}$. Note that every subset of vertices which contains at least $\frac{n}{2}+1$ vertices is a dominating set as it necessarily contains vertices from both parts. Therefore $d_i(G)={n \choose i}$ for all $i \geq \frac{n}{2}+1$. Furthermore $d_i(G)$ is non-increasing for $i \geq \frac{n}{2}+1$ and hence $G$ is unimodal.
\end{proof}

\section{Almost all graphs are unimodal}

In this section we will show that the domination polynomial of almost all graphs is unimodal with mode $\lceil \frac{n}{2} \rceil$, and hence that any counterexamples to unimodality are relatively rare. 


%
%
%
We will now show graphs with minimum degree at least $2\log_2(n)$ are unimodal. We begin with a few preliminary definitions and observations. For a graph of order $n$, let $r_i(G)$ proportion of subsets of vertices of $G$ with cardinality $i$ which are dominating. That is,

$$r_i(G) = \frac{d_i(G)}{{ n \choose i }}.$$

\noindent Note that $0 \leq r_i(G) \leq 1$. For all $1 \leq i \leq n$, let $\mathcal{D}_i(G)$ denote the collection of dominating sets of cardinality  exactly $i$. Note for any dominating set $S \in \mathcal{D}_{i}(G)$ and any vertex $v \in V-S$, $S \cup \{v\} \in \mathcal{D}_{i+1}(G)$. More specifically if we let $A_{i+1} = \{(v,S):S \in  \mathcal{D}_{i+1}(G), v \in S\}$ and $B_i = \{(v,S):S \in  \mathcal{D}_{i}(G), v \notin S\}$ there is an injective mapping $f:B_i \rightarrow A_{i+1}$ defined as $f(v,S)=(v,S\cup \{v\})$. Therefore $|A_{i+1}| \geq |B_{i}|$ and equivalently $(i+1)d_{i+1}(G) \geq (n-i)d_{i}(G)$. Furthermore

$$r_{i+1}(G) = \frac{d_{i+1}(G)}{{ n \choose i+1 }} \geq \frac{(n-i)d_{i}(G)}{(i+1){ n \choose i+1 }} = \frac{d_{i}(G)}{{ n \choose i}}=r_{i}(G).$$

\noindent This allow us to obtain the following lemma.

\begin{lemma}
\label{lem:riffincrease}
Let $G$ be a graph on $n$ vertices, and $k \geq \frac{n}{2}$. If $r_{k}(G) \geq \frac{n-k}{k+1}$ then $d_{i+1}(G) \leq d_{i}(G)$ for all $i \geq k$. In particular if $k= \lceil \frac{n}{2} \rceil$ then $G$ is unimodal with mode $ \lceil \frac{n}{2} \rceil$.
\end{lemma}

\begin{proof}
Set $d_i=d_{i}(G)$ and $r_i=r_{i}(G)$ for all $i$. Note that

$$d_{i+1} \leq d_{i} \hspace{2mm} \Leftrightarrow  \hspace{2mm} r_{i+1}{n \choose i+1} \leq r_{i}{n \choose i} \hspace{2mm} \Leftrightarrow \hspace{2mm} \frac{r_{i+1}}{r_i} \leq \frac{i+1}{n-i} \hspace{2mm} \Leftrightarrow \hspace{2mm} \frac{r_{i}}{r_{i+1}} \geq \frac{n-i}{i+1}.$$

\noindent Therefore for each $i$, if $r_{i} \geq \frac{n-i}{i+1}$ then $d_{i+1} \leq d_{i}$ as $r_{i+1} \leq 1$. So suppose for some $k \geq \frac{n}{2}$, $r_{k}(G) \geq \frac{n-k}{k+1}$. Then for any $i \geq k$ we have 

$$r_{i}(G) \geq r_{k}(G) \geq \frac{n-k}{k+1} \geq \frac{n-i}{i+1}$$

\noindent and hence $d_{i+1} \leq d_{i}$. Finally, if $k= \lceil \frac{n}{2} \rceil$ then together with Proposition~ \ref{thm:domincrease} we have

$$d_1 \leq d_2 \leq \cdots \leq d_{\lceil \frac{n}{2} \rceil} \geq \cdots \geq d_n.$$

\end{proof}

\begin{theorem}
\label{thm:2lnnunimodal}
If $G$ is a graph with $n$ vertices with minimum degree $\delta(G) \geq 2\log_2(n)$ then $D(G,x)$ is unimodal with mode at $\lceil \frac{n}{2} \rceil$.
\end{theorem}

\begin{proof}
Set $\delta=\delta(G)$, $d_i=d_{i}(G)$ and $r_i=r_{i}(G)$ for all $i$. Let $n_i$ denote the number of non-dominating subsets $S \subseteq V(G)$ of cardinality $i$. Note that $n_i = { n \choose i} - d_i$ and hence

$$r_i = 1- \frac{n_i}{ { n \choose i}}.$$

\noindent We will now show $n_i \leq n {n-\delta-1 \choose i}$. For each vertex $v \in V$ let $n_i(v)$ denote the number of subsets sets which do not dominate $v$. A subset $S$ does not dominate $v$ if and only if it does not contain any vertices in $N[v]$. Therefore $n_i(v)$ simply counts every subset of $V(G)$ with $i$ vertices which omits $N[v]$. Hence $n_i(v)= {n-\deg(v)-1 \choose i}$. Furthermore any non-dominating set of order $i$ must not dominate some vertex of $G$. Therefore

$$n_i \leq \sum_{v \in V}n_i(v) = \sum_{v \in V}{n-\deg(v)-1 \choose i} \leq \sum_{v \in V}{n-\delta-1 \choose i} = n{n-\delta-1 \choose i},$$

\noindent and

\vspace{-6mm}

\begin{align*}
r_i =&  1- \frac{n_i}{ { n \choose i}} \\
    \geq&  1- \frac{n{n-\delta-1 \choose i}}{ { n \choose i}} \\
    \geq&  1- \frac{n (n-\delta-1)!}{i!(n-\delta-1-i)!} \cdot \frac{i!(n-i)!}{n!}\\ 
    \geq&  1- \frac{(n-1-\delta)!}{(n-1)!} \cdot \frac{(n-i)!}{(n-i-\delta-1)!}\\  
    \geq&  1- \frac{(n-i)(n-i-1)(n-i-2) \cdots (n-i-\delta)}{(n-1)(n-2) \cdots (n-\delta)}.
\end{align*}

\noindent Note that for any $k \geq 0$, $\frac{n-i-k}{n-k} \geq \frac{n-i-k-1}{n-k-1}$ 
holds as $i \geq 0$. Therefore

$$\frac{n-i}{n} \geq \frac{n-i-1}{n-1} \geq \cdots \geq \frac{n-i-\delta}{n-1-\delta}. $$

\noindent and so 

$$r_i \geq 1-(n-i) \left( \frac{n-i}{n} \right)^\delta.$$

\noindent Now let $f(x,\delta) = 1-(n-x) \left( \frac{n-x}{n} \right)^\delta$ and $g(x)=\frac{n-x}{x+1} = \frac{n+1}{x+1}-1$ for $x, \delta \in [0,n]$. Note that $f(x,\delta)$ is an increasing function of both $x$ and $\delta$ and $g(x)$ is also a decreasing function of $x$. By Lemma \ref{lem:riffincrease}, it suffices to show $f(\frac{n}{2},2\log_2(n)) \geq g(\frac{n}{2})$. Note

$$f\left(\frac{n}{2},2\log_2(n)\right) = 1-\frac{n}{2} \left( \frac{1}{2} \right)^{2\log_2(n)}=1-\frac{n}{2n^2}=1-\frac{1}{2n}$$

\noindent and 

$$g\left(\frac{n}{2}\right) = \frac{\frac{n}{2}}{\frac{n}{2}+1} = \frac{n}{n+2} = 1-\frac{2}{n+2}.$$

\noindent Therefore $f(\frac{n}{2},2\log_2(n)) \geq g(\frac{n}{2})$ if and only if $\frac{2}{n+2} \geq \frac{1}{2n}$ which holds for all $n\geq 1$.
\end{proof}

Let ${\mathcal G}(n,p)$ denote the Erd\"{o}s-R\'{e}nyi random graph model on $n$ vertices (each edge exists is independent present with probability $p$). 

\begin{theorem}
\label{thm:almostall}
Fix $p \in (0,1)$. Let $G_{n} \in {\mathcal G}(n,p)$. Then with probability tending to $1$, $D(G_n,x)$ is unimodal with mode $\lceil \frac{n}{2} \rceil$.
\end{theorem}

\begin{proof}
The degree of any vertex $v$ of $G_n$ has a binomial distribution $X_v$ with $N = n-1$, with mean $p(n-1)$. From Hoeffding's well known bound on the tail of a binomial distribution, it follows that for any fixed $\varepsilon > 0$,
\[ \mbox{Prob}\left( X_v \leq (p-\varepsilon)(n-1) \right) \leq e^{-2\varepsilon^2(n-1)}.\]
Thus 
\[ \mbox{Prob}\left( \cup_{v} X_v \leq (p-\varepsilon)(n-1) \right) \leq ne^{-2\varepsilon^2(n-1)} \rightarrow 0.\]
It follows that for sufficiently large $n$, $\delta(G_n) > (p-\varepsilon)(n-1) > 2 \log_2(n)$ with probability tending to $1$.
By Theorem \ref{thm:2lnnunimodal}, it follows that, with probability tending to $1$, $D(G_n,x)$ is unimodal with mode $\lceil \frac{n}{2} \rceil$.

\end{proof}

\section{Open Problem}

Theorem \ref{thm:2lnnunimodal} shows Conjecture \ref{conj:unimodal} is true for graphs with sufficiently high minimum degree. However, the conjecture remains elusive for graphs with low minimum degree, and in particular for trees. Another interesting family of graphs to investigate are graphs with universal vertices. We verified using Maple all graphs of order up to 10 which have universal vertices are unimodal, with mode at either $\lceil \frac{n}{2} \rceil$ or $\lceil \frac{n}{2} \rceil+1$. If a graph $G$ with $n$ vertices has a universal vertex then $d_i(G) \geq {n-1 \choose i-1}$ and hence $r_i(G) \geq \frac{i}{n}$. It is possible a technique similar to the one used in Theorem \ref{thm:2lnnunimodal} can yield some results for this class.

From a well known theorem of Newton, if a polynomial $f$ with positive coefficients has all real roots then $f$ is log-concave and hence unimodal, and Darroch \cite{1964Darroch} further showed that mode of such an $f$ is at either $\left\lfloor \frac{f'(1)}{f(1)} \right\rfloor$ or $\left\lceil \frac{f'(1)}{f(1)} \right\rceil$. In \cite{2020Avd} the authors defined the \emph{average size of a dominating set} in a graph $G$ as $\avd(G) = \frac{D'(G,1)}{D(G,1)}$. They also showed $\frac{n}{2} \leq \avd(G) \leq \frac{n+1}{2}$ for graphs with minimum degree $\delta \geq 2\log_2(n)$. Theorem \ref{thm:2lnnunimodal} implies that the mode of $D(G,x)$ is at $\lceil \avd(G) \rceil$ or $\lfloor \avd(G) \rfloor$ for graphs with minimum degree $\delta \geq 2\log_2(n)$. This leads us to the following question: For a graph $G$, is the mode of $D(G,x)$ always at $\lceil \avd(G) \rceil$ or $\lfloor \avd(G) \rfloor$? If not, how much can the mode and $\avd(G)$ differ?

Finally, Proposition~\ref{thm:domincrease} shows that up the half-way mark, the coefficients of the domination polynomial are non-decreasing, so that any problem with unimodality must occur after this point. We can show that if a graph has no isolated vertices, then from $\lfloor \frac{3n}{4} \rfloor$, the coefficients are non-increasing. It would certainly be worthwhile to investigate further the last half of the coefficients sequence for graphs with isolated vertices, and the middle quarter for those that do not. 

\section*{Acknowledgements}
 
\noindent J. Brown acknowledges research support from Natural Sciences and Engineering Research Council of Canada (NSERC), grants RGPIN 2018-05227.


\bibliographystyle{plain}
\bibliography{MyBibFile}

\begin{thebibliography}{10}

\bibitem{2012AlikhaniPHD}
S.~Alikhani.
\newblock {\em {Dominating sets and domination polynomials of graphs}}.
\newblock Lambert Academic Publishing, first edition, 2012.

\bibitem{IntroDomPoly2014}
S.~Alikhani and Y.~H. Peng.
\newblock {Introduction to domination polynomial of a graph}.
\newblock {\em Ars Comb.}, 114:257--266, 2014.

\bibitem{2014DomFamUnimodal}
Saeid Alikhani and Somayeh Jahari.
\newblock {Some families of graphs whose domination polynomials are unimodal}.
\newblock {\em Iran. J. Math. Sci. Informatics}, 12(1):69--80, 2014.

\bibitem{2020Avd}
Iain Beaton and Jason~I Brown.
\newblock {The Average Order of Dominating Sets of a Graph}.
\newblock {\em Submitted}, pages 1--21, 2020.

\bibitem{2007Chudnovsky}
Maria Chudnovsky and Paul Seymour.
\newblock {The roots of the independence polynomial of a clawfree graph}.
\newblock {\em J. Comb. Theory, Ser. B}, 97:350--357, 2007.

\bibitem{1964Darroch}
J.~N. Darroch.
\newblock {On the Distribution of the Number of Successes in Independent
  Trials}.
\newblock {\em Ann. Math. Stat.}, 35(3):1317--1321, 1964.

\bibitem{1970Corona}
R.~Frucht and F.~Harary.
\newblock {On the corona of two graphs}.
\newblock {\em Aequationes Math}, 4:322--324, 1970.

\bibitem{1990Hamidoune}
Yahya~Ould Hamidoune.
\newblock {On the Numbers of independent k-Sets in a Claw Free Graph}.
\newblock {\em J. Comb. Theory, Ser. B}, 50:241--244, 1990.

\bibitem{hedetniemi}
T.W. Haynes, S.~Hedetniemi, and P.J. Slater.
\newblock {\em {Fundamentals of Domination in Graphs}}.
\newblock Marcel Dekker, 1998.

\bibitem{1972Heilmann}
O.J. Heilmann and E.H. Lieb.
\newblock {Theory of Monomer-Dimer Systems}.
\newblock {\em Commun. Math. Phys.}, 25(3):190--232, 1972.

\bibitem{2002Horrocks}
David G~C Horrocks.
\newblock {Note: The Numbers of Dependent k-Sets in a Graph Are Log Concave}.
\newblock {\em J. Comb. Theory, Ser. B}, 84:180--185, 2002.

\bibitem{2012Recurr}
T.~Kotek, J.~Preen, F.~Simon, P.~Tittmann, and M.~Trinks.
\newblock {Recurrence relations and splitting formulas for the domination
  polynomial}.
\newblock {\em Electron. J. Comb.}, 19:1--27, 2012.

\bibitem{1996Krattenthaler}
C~Krattenthaler.
\newblock {Note Combinatorial Proof of the Log-Concavity of the Sequence of
  Matching Numbers}.
\newblock {\em J. Comb. Theory, Ser. A}, 74:351--354, 1996.

\end{thebibliography}

\end{document}